\documentclass[a4paper, 11pt]{article}
\usepackage[utf8]{inputenc}
\usepackage{geometry}
\usepackage{graphicx}
\usepackage{authblk} %
\usepackage{cancel} %
\usepackage{amsmath}
\usepackage{amssymb}
\usepackage{amsthm}
\numberwithin{equation}{section}
\usepackage{stmaryrd} %

\usepackage{xcolor}%
\usepackage{mathtools}

\usepackage{csquotes} %
\usepackage{hyperref}
\hypersetup{colorlinks,citecolor=blue,filecolor=blacklinkcolor=red, urlcolor=magenta,linkcolor=red}

 \usepackage[normalem]{ulem} %
\usepackage{tikz,tikz-cd} %

\newtheorem{defi}{Definition}[section]
\newtheorem{lemm}{Lemma}[section]
\newtheorem{coro}{Corollary}[section]
\newtheorem{theo}{Theorem}[section]

\newtheorem{fact}{Fact}[section]
\newtheorem{prop}{Proposition}[section]
\theoremstyle{remark}
\newtheorem{rema}{Remark}[section]

\newtheoremstyle{myexample}%
{3pt}%
{3pt}%
{}%
{}%
{\bfseries}%
{}%
{.5em}%
{}%
\theoremstyle{myexample}
\newtheorem{exam}{Example}[section]
\usepackage[maxbibnames=10,articlein=false,giveninits=true,style=ext-numeric,hyperref,backend=biber]{biblatex}
\AtEveryBibitem{%
  \clearfield{issn} %
  \ifentrytype{online}{}{%
    \clearfield{url}
  }
}

\newcommand{\overbar}[1]{\mkern 1.5mu\overline{\mkern-1.5mu#1\mkern-1.5mu}\mkern 1.5mu}
\DeclareMathOperator{\sh}{sh}
\DeclareMathOperator{\ch}{ch}

\DeclareMathOperator{\id}{id}
\DeclareMathOperator{\Ad}{Ad}
\DeclareMathOperator{\vspan}{span}
\newcommand{\dd}{\textnormal{d}}
\newcommand{\N}{\mathbb{N}}
\newcommand{\Z}{\mathbb{Z}}
\newcommand{\R}{\mathbb{R}}
\newcommand{\C}{\mathbb{C}}
\newcommand{\Mbar}{\overbar{M}}

\newcommand{\pfrac}[2]{\frac{\partial#1}{\partial#2}}

\newcommand{\IsoHyperplane}{\tilde{G}} %
\newcommand{\RP}{\R\textnormal{P}}
\newcommand{\PGL}{\textnormal{PGL}}
\newcommand{\double}{D(\Mbar)}

\title{Boundary rigidity, and non-rigidity, of projective structures} 
\author{Jack Borthwick\thanks{\href{mailto:jack.borthwick@math.cnrs.fr}{jack.borthwick@mcgill.ca}}, \, Niky Kamran\thanks{\href{mailto:niky.kamran@mcgill.ca}{niky.kamran@mcgill.ca}}}
\affil{Department of Mathematics and Statistics, McGill University\\}

\begin{document}
\maketitle
{\let\thefootnote\relax}
\begin{abstract}
We investigate the property of boundary rigidity for the projective structures associated to torsion-free affine connections on connected smooth manifolds with boundary. We show that these structures are generically boundary rigid, meaning that any automorphism of a generic projective structure that restricts to the identity on the boundary must itself be the identity. However, and in contrast with what happens for example for conformal structures, we show that there exist projective structures which are not boundary rigid. We characterise these non-rigid structures by the vanishing of a certain local projective invariant of the boundary. 
\end{abstract}

\section{Introduction}
We say that a geometric structure defined on a smooth manifold with boundary is {\em{boundary rigid}} if, given any automorphism of the structure that restricts to the identity on the boundary, that automorphism's Taylor expansion centered at any boundary point reduces to that of the identity. This means in particular that in the analytic category, a non-trivial automorphism of a boundary rigid geometric structure cannot restrict to the identity on the manifold's boundary; in this case we will say that the manifold has a \emph{rigid boundary}. (See Definition~\ref{DefRigidBoundary} below.)

Conformal structures on Riemannian manifolds with boundary constitute an important class of boundary rigid structures, as shown in~\cite[Proposition 3.3]{WRBLionheart_1997}. It is appropriate to mention at this stage that the boundary rigidity of conformal structures is a significant element in the study of the anisotropic Calder\'on inverse problem, which is to recover the metric of a Riemannian manifold with boundary from the Dirichlet-to-Neumann map for the Laplace-Beltrami operator. The specific role of boundary rigidity appears in this context through the invariance of the Dirichlet-to-Neumann map under diffeomorphisms that restrict to the identity on the boundary, meaning that the group of such diffeomorphisms constitutes a gauge invariance for the anisotropic Calder\'on problem. In particular, the boundary rigidity of conformal structures is an important ingredient in the construction of counterexamples to uniqueness for the Calder\'on inverse problem~\cite{Daude:2020aa,Daude:2019aa}.

From a geometric point of view, conformal structures are a special case of an important subclass of Cartan geometries known as parabolic geometries~\cite{Cap:2009aa}, which have received considerable attention over the past decade. It is thus natural to ask to what extent the property of boundary rigidity may hold within the broader class of parabolic geometries.  Our paper makes a first foray into this question by investigating the question of boundary rigidity for \emph{projective structures }on connected manifolds with boundary, that is, the structures defined by the set of unparametrised geodesics of torsion-free affine connections. We shall see that the situation is markedly different when compared with the conformal case in that there exist projective structures which are \emph{not} boundary rigid and which can be characterised invariantly. 
The question of boundary rigidity for projective structures also has inherent interest. For instance, it is relevant in the study of the long time dynamics and asymptotic behaviour of solutions to partial differential equations on manifolds (without boundary) equipped with a complete projectively compact affine connection as defined in~\cite{Cap:2016aa}, see also~\cite{Borthwick:2021aa}. It could be expected that some aspects of the asymptotic analysis could be recast into boundary value problems at the boundary and projective transformations would then be natural symmetries to consider.

One of the first ingredients in the proof of the boundary rigidity of conformal structures is that any conformal geometry is equivalent to a unique normal Cartan geometry, of which it is well known that the automorphism group is a finite-dimensional Lie group~\cite[Chapter IV, Theorem 6.1]{Kobayashi:1995aa}. Roughly speaking, this expresses the fact that in local coordinates the automorphisms of the structure satisfy some generically non-linear overdetermined system of partial differential equations whose solutions are parametrised by arbitrary constants (these constants provide a set of local coordinates on the Lie group). We shall recall shortly an analogous statement for projective structures (see Fact~\ref{DimensionGroupProjectiveTransformations}).

In contrast with conformal structures, the basic geometric data of a projective structure consist as mentioned above in the underlying trajectories of the geodesic flow of a spray or, in other words, the set of \emph{unparametrised} geodesics of a torsion-free affine connection $\nabla$. A projective structure %
on a manifold (without boundary) $M$ is thus often described as a class $[\nabla]$ of \emph{projectively equivalent} torsion-free affine connections, i.e. a class of torsion-free affine connections with the same geodesics up to reparametrisation. This property is characterised by the following classical result (see for instance~\cite[Proposition 7.2]{Kobayashi:1995aa}):
\begin{fact}\label{FactProjectiveEquivalence}
A pair of torsion free affine connections ${\nabla}$ and $\hat{\nabla}$ are projectively equivalent if and only if there exists a $1$-form $\Upsilon$ such that for all vector fields $X,Y$ on $M$, we have:
\[ \hat{\nabla}_X Y = \nabla_X Y + \Upsilon(X)Y + \Upsilon(Y)X. \]
\end{fact}

In these terms, a projective transformation is a diffeomorphism of $\phi: M\rightarrow M$ such that for any connection $\nabla$ in the projective class, $\phi^*\nabla$ and $\nabla$ are projectively equivalent.

Similarly to the conformal case, it is a well-known result of E. Cartan~\cite{Cartan:1924aa,Kobayashi:1995aa,Sharpe:1997aa} that classes of projectively equivalent torsion-free affine connections are in one-to-one correspondence with normal Cartan geometries on $M$ modeled on projective spaces. One can then show that projective transformations are precisely the automorphisms of the corresponding normal Cartan geometry so that by~\cite[Chapter IV, Theorem 6.1]{Kobayashi:1995aa}, one has:
\begin{fact} \label{DimensionGroupProjectiveTransformations}The group of projective transformations is a Lie group at most of dimension $n(n+2)=\dim \PGL(n,\R)$.
\end{fact}

A further, equivalent, point of view due to Thomas~\cite{Thomas:1934aa,Bailey:1994aa}, which in modern-day language is formulated in terms of vector bundles known as tractor bundles~\cite{Bailey:1994aa,Cap_Gover_2002}, is based on the observation that if $(\alpha^i_j)$ is the $\mathfrak{gl}(n;\R)$-valued connection form of a connection $\nabla$ in a projective class on an $n$-dimensional manifold $M$, expressed in an arbitrary local frame $(e_i)$ (with dual frame ($\omega^i))$, then the form:
\begin{equation}\label{ThomasInvariant} \Pi^i_j=\alpha^i_j -\frac{1}{n+1}\alpha^k_k \delta^i_j - \frac{1}{n+1}\alpha^k_k(e_j)\omega^i  \end{equation}
is a projective invariant, i.e. independent of the choice of $\nabla$. We will refer to $(\Pi^i_j)$ as the \emph{(local) Thomas (projective) invariant}. As Thomas observed, $(\Pi^i_j)$ does not transform under diffeomorphisms like an affine connection on the frame bundle, and as such is not a distinguished choice in the class, but can be used to define a unique linear connection on a vector bundle of rank $n+1$; this linear connection being equivalent to the normal Cartan connection.  The form $\Pi$ can in fact be interpreted as the $\mathfrak{gl}(n;\R)$-component of the normal Cartan connection in an appropriate Cartan gauge, and with the $\R^n$-valued one form $\theta=(\omega^i)$ it determines the normal connection uniquely. 

For the study of projective transformations the local invariant $\Pi$ provides the following local characterisation of projective transformations: \begin{lemm}\label{ThomasProjectiveTransformation} A diffeomorphism $\phi : M \rightarrow M$ is a projective transformation if and only if for any local sections $\sigma_U : U \rightarrow P^1(M)$, $\sigma_V : V \rightarrow P^1(M)$ of the frame bundle defined near $p$ and $\phi(p)$ respectively the local invariants $\Pi_U$ and $\Pi_V$ defined by~\eqref{ThomasInvariant} satisfy:  \[ (\phi|_U)^*\Pi_V = \Pi_U.\]
\end{lemm}
We will elaborate on these points further in Section~\ref{Projective_Structures_On_MWB} where we define projective structures on manifolds with boundary and briefly justify that the above discussion passes without any essential modification to this setting.

We now formulate the concept of boundary rigidity for projective structures. 
\begin{defi}\label{DefRigidBoundary}
We shall say that the boundary $\partial M$ of a manifold with boundary $\Mbar$ endowed with a projective structure (defined in Section~\ref{Projective_Structures_On_MWB}) is \emph{rigid} if the only projective transformation that restricts to the identity map on the boundary is the identity map. 
\end{defi}

Fact~\ref{DimensionGroupProjectiveTransformations}, which is consequence of a general statement for the automorphism group of an arbitrary Cartan geometry, is the first hint that boundary rigidity may not be limited to conformal structures. Indeed, one could expect that the condition of restricting to the identity map on the boundary exhausts a large number of degrees of freedom. It may be interesting to note that for conformal structures the maximal dimension of the automorphism group is $\frac{1}{2}(n+1)(n+2)\leq n(n+2)$.

The main result of our work is that the boundary of a smooth manifold with boundary endowed with a smooth projective structure is generically rigid in the above sense. More surprisingly however, and contrary to the conformal case, there exists a class of projective structures for which we have \emph{non-rigidity}. We shall discuss this further below, but first, we shall state our rigidity result after introducing our notation.

Throughout the paper we will work on a smooth manifold with boundary $\Mbar$ of dimension $n\geq 2$.  The interior will be denoted by $M$ and the boundary by $\partial M$. 

We will make extensive use of boundary charts with coordinates $(r,y^1,\dots, y^{n-1})$. From this it should be understood that $r$ is a boundary defining function on the defining neighbourhood $U$ of the chart, this means that $\partial M \cap U = \{r=0\}$ and $\dd r$ is nowhere vanishing on $\partial M \cap U$. 

When convenient, we will set $x^0=r$ and $x^\mu=y^\mu$ for any $1\leq \mu \leq n$. It will be understood that Latin indices, e.g. $i,j,k,\dots,$ will range from $0$ to $n-1$ and Greek indices $\mu,\nu,\sigma, \dots$ from $1$ to $n-1$. 
Symmetrisation of indices will be denoted as usual by round brackets and anti-symmetrisation by square brackets.

In Section~\ref{NonRigidity}, we make a brief but convenient use of the Penrose abstract index notation~\cite{Penrose:1984aa}; to avoid confusion, we reserve the beginning of the Latin alphabet, $a,b,c,\dots,$ for these indices. In brief, the principle is that abstract indices are labels of the nature of the tensor quantities, rather than indices with numerical value, and do not refer to any choice of frame. For instance $X^a$ would indicate that $X$ is a vector field, $u_a$, that $u$ is a $1$-form. Contraction is therefore simply written as $X^au_a=u(X)$. 

We now state 
\begin{theo}\label{ThmRigidity}
Let $\Mbar$ be a connected smooth manifold with boundary endowed with a smooth projective structure (see Section~\ref{Projective_Structures_On_MWB}). Suppose that there exists a boundary point $p\in \partial M$ and a local boundary chart near $p$ such that the local Thomas invariant form satisfies $\Pi^0_\nu(p) \not\in\vspan(\dd r(p))$ for at least one $1\leq \nu \leq n-1$. Then the boundary $\partial M$ is rigid.
\end{theo}

The condition $\Pi^0_\nu(p) \notin\vspan (\dd r(p))$ appearing in the above theorem may seem a little strange at first glance. Indeed, one may expect that it should always be possible to choose a coordinate chart  %
such that it is satisfied. We respond to this by the negative:

\begin{prop}\label{PropProjectiveAndCoordinateInvariance}
If there is a point $p\in \partial M$ and a boundary chart such that:
$\Pi^0_\nu(p) \in\vspan (\dd r_p)$ for all $1\leq \nu \leq n-1$, then this holds in any boundary chart near $p$.
\end{prop}

It is of course essential that the point $p\in \partial M$ be a boundary point; Proposition~\ref{PropProjectiveAndCoordinateInvariance} fails at interior points. The above result highlights however that the above condition expresses something about the geometry of the boundary. One may conjecture that the condition that the hypothesis of Proposition~\ref{PropProjectiveAndCoordinateInvariance} hold at every boundary point, as required for non-rigidity, is strong enough to determine a projective structure on the boundary; we found that this is indeed the case. In fact, this projective structure is obtained from another naturally induced Cartan geometry on the boundary. The model for this second Cartan geometry is surprisingly naïve: it is the structure one obtains on a hyperplane $\mathcal{H}$ by reducing the projective group $\PGL(n,\R)$ to the subgroup $\IsoHyperplane$ that stabilises it. It also has the curious feature, that it is \emph{not} an effective Cartan geometry (see~\cite[Chapter \textbf{4}\S3] {Sharpe:1997aa}), for the action of $\tilde{G}$ on the hyperplane is not effective. We elaborate on this further in Section~\ref{NonRigidity}.

\subsection*{Acknowledgements}
Research supported by NSERC grant RGPIN 105490-2018.

The authors would like to thank the anonymous referee for their valuable comments and questions that led to the discussion in Section~\ref{Projective_Structures_On_MWB}.

\section{Projective structures on manifolds with boundary}\label{Projective_Structures_On_MWB} 
{Projective structures on manifolds with boundary have appeared before in the literature, for instance, in the works of Čap and Gover on projective compactification~\cite{Cap:2014aa,Cap:2016aa}, however there does not seem to be any comprehensive account of their general theory. To fill this small gap in the literature,} we shall give a definition in the present section. It is important to emphasise from the outset that, throughout, we abide by the standard convention that at boundary points, the tangent space remains a \emph{vector} space isomorphic to $\R^n$; it is \emph{not} a half-space. We refer to~\cite{Lee:2003aa} for a textbook discussion on this point. Other pointwise constructions will follow a similar principle making boundary points, at this level, indistinguishable from other points. For instance, with the understanding that the tangent space $T_p\Mbar$ at a boundary point $p$ is a vector space, it is natural that the set of frames at $p$ will be isomorphic to $GL(\R^n, T_p\Mbar)$. In particular, frames at boundary points will not necessarily be compatible with any pointwise direct sum decomposition of $T_p\Mbar$ in which one of the summands would be $T_p\partial M$, where $\partial M$ is viewed as an embedded submanifold of $\Mbar$.

\subsection{Frame bundles on manifolds with boundary}
The first step is to extend the definition of the $k$-th order frame bundles to manifolds with boundary.  For this we will follow the approach in~\cite[\S 2.12]{Michor2011} for manifolds with corners and define a $k$-frame at $x\in \Mbar$ to be the equivalence class $[j^k_0f]$ of the $k$-jet of a local diffeomorphism $f: U \rightarrow \tilde{V}$ between neighbourhood $U$ of $0\in \R^n$ and an open subset $\tilde{V}$ of a manifold \emph{without} boundary $N$, containing the image of a neighbourhood $V$ of $x\in \Mbar$ under a (local) embedding $\psi: V\to \tilde V \subset N$, where $f(0)=\psi(x)$. In the above the equivalence relation is defined as follows: the $k$-jet of $f: U\rightarrow \tilde{V}$ defines the same $k$-frame as the $k$-jet of $f':U' \rightarrow \tilde{V}'$, $\tilde{V}'\subset N'$ if the coefficients of their Taylor polynomials to order $k$ in local charts are related by the chain rule. 

One can remark that a $k$-frame has a representative such that $\tilde{V}$ is a subset of $\R^n$ and the local embedding $\psi$ is a coordinate chart, $x: V \rightarrow x(V) \subset \tilde{V} $; the partial derivatives of this representative can be used to define a canonical coordinate system $(x^i,u^i_j,u^i_{jk})$ associated with the chart $(x,V)$.
With this in mind, the set of $k$-frames $P^k(\Mbar)$ can be given a topology and a smooth manifold with boundary structure just as in the way in which one proceeds for manifolds without boundary. In this case, however, $P^k(\Mbar)$ is itself a manifold \emph{with} boundary. More precisely, if $r : \Mbar \rightarrow \R_+$ is a local boundary defining function on $\Mbar$ then $r\circ \pi$ is a local boundary defining function of $P^k(\Mbar)$ where $\pi : P^k(\Mbar)\longrightarrow\Mbar$ is the canonical projection map. If $\partial M = \emptyset$ then this definition coincides with the usual one as one can always choose $N=M$. In particular, the interior of $P^k(\Mbar)$ can be identified with $P^k(M)$. 
\begin{rema}
One may also observe that it is possible to choose $N=D(\Mbar)$, the double of $\Mbar$ (see~\cite[Example 9.32 p. 226]{Lee:2003aa}), and $\psi$ an embedding of $\Mbar$ into $\double$. An equivalent approach could have been to define $P^k(\Mbar)=\psi^*P^k(\double)$, but the construction we have described above is manifestly independent of the choice of local extension of $\Mbar$. 
\end{rema}

If $\phi : \Mbar \rightarrow \Mbar$ is a diffeomorphism of $\Mbar$, then $\phi$ induces a diffeomorphism $\phi^{(k)}$ of $P^k(\Mbar)$, its $k$-th prolongation, according to the definition: $\phi^{(k)}([j^k_0f])=[j^k_0(\phi\circ f)]$ after locally extending $\phi$, when $f(0)=p$ is a boundary point, to a smooth map defined on a neighbourhood $\tilde{V}$ of a manifold without boundary that contains the embedded image of a neighbourhood $V\subset \Mbar$ of $p$, which is possible by definition of smoothness; the $k$-frame is easily seen to be independent of the extension. 

To discuss projective structures we are particularly interested in the cases $k\in \{1,2\}$. We recall that the set of $2$-frames at $0\in \R^n$ is a group denoted by $G^2(\R^n)$ with multiplication defined by:
\[(u^i_j, u^i_{jk})\cdot (s^i_j, u^i_{jk})=(u^i_ks^k_j, u^i_{lm}s^l_js^m_k+ u^i_ms^m_{jk}),\]
this acts on $P^2(\Mbar)$ from the right. This right-action equips $\pi : P^2(\Mbar) \rightarrow \Mbar$ with the structure of a $G^2(\R^n)$-principal bundle; extending verbatim the standard definition to manifolds with boundary. In other words, the local model is given by $\mathbb H^n\times G^2(\mathbb R^n)\to \mathbb H^n$.

The definition of the canonical form $\theta$ on $P^2(\Mbar)$ from~\cite[\S 5, Chapter IV]{Kobayashi:1995aa} can be extended to manifolds with boundary as follows.  Let $f: U \rightarrow \tilde{V}\subset N$, $\psi : V \rightarrow N$ represent a $2$-frame on $\Mbar$, then $\psi$ induces a local embedding of $P^1(\Mbar)$ in $P^1(N)$. Now, the first prolongation $f^{(1)}$ of $f$ induces a local diffeomorphism $\bar f$ between $T_{(0,I_n)}P^1(\R^n)= \mathfrak{a}(n;\R)$ and $T_{j^1_0f} P^1(N)$. Since $P^1(\Mbar)$ and $P^1(N)$ have same dimension, it follows that the embedding $\psi$ induces a linear isomorphism $w:T_{[j^1_0f]} P^1(\Mbar) \rightarrow T_{j^1_0 f}P^1(N)$. Expressing this fact in coordinates, one observes that the isomorphism:
$\tilde f: \mathfrak{a}(n;\R) \rightarrow T_{[j^1_0f]}P^1(\Mbar)$ defined by $\tilde f = \bar f\circ w^{-1}$, depends only on the $2$-frame that $f$ determines at $x\in \Mbar$, hence we define, for any $X\in T_{[j^2_0f]}P^2(\Mbar)$:
\[ \theta(X)=\tilde{f}^{-1}(p_{*}X), \]
where $p: P^2(\Mbar) \rightarrow P^1(\Mbar)$ is the canonical projection.

\subsection{Projective structures}\label{SectionDefinitionProjectiveStructure}
In order to introduce the definition of a projective structure on a manifold with boundary, let us first fix some conventions. Projective space is a homogeneous manifold $G/H$ with $G=\PGL(n;\R)=GL(n+1;\R)/Z$ where $Z$ is the centre of $GL(n+1;\R)$ consisting of the matrices $\{\lambda I_{n+1}, \lambda \neq 0\}$ and $H$ is the isotropy subgroup of the point with homogenous coordinates ${}^t\begin{pmatrix}0_{\R^n}&1\end{pmatrix}$:
\begin{equation}\label{StabiliserPointProjectiveGeometry} H=\left\{ \begin{pmatrix} A & 0 \\ \Upsilon & a \end{pmatrix}\in GL(n+1;\R)/Z \right\}. \end{equation}
An important fact shown in~\cite[\S6,Chapter IV]{Kobayashi:1995aa} is that the subgroup $H$ can be realised as the subgroup of the group $2$-frames at $0\in \R^n$, $G^2(\R^n)$ given by:
\[\left\{ \left(A^{i}_{j},-2A^i_{(j}\Upsilon_{k)} \right)\right\} \subset G^2(\R^n).\]
We define:
\begin{defi}\label{DefinitionProjectiveStructure}
A projective structure on $\Mbar$ is a global section $s$ of $q: P^2(\Mbar)/H \rightarrow \Mbar$.
\end{defi}
If $(x^i, u^i_j, u^i_{jk})$ are some canonical coordinates on $P^2(\Mbar)$ associated with a chart on $\Mbar$, then one can define local coordinates on $P^2(\Mbar)/H$ such that the projection map $q$ is represented by:
\[(x^i, u^i_j, u^i_{jk}) \mapsto (x^i, u^i_{lm}v^l_jv^m_k -\frac{2}{n+1} \delta^i_{(j}v^m_{k)}u^s_{lm}v^l_s),\]
where $u^i_kv^k_j=\delta^i_j$. Note that the second component is trace-free and symmetric in its lower indices.
It follows that a section $s$ is given locally by functions $x \mapsto z^{i}_{jk}(x)$. Under a change of chart $\phi=y\circ x^{-1}$, these transform according to
\[\tilde{z}^i_{jk} = \phi^i_{lm}(\phi^{-1})^l_j(\phi^{-1})^m_k + \phi^i_a z^a_{lm}(\phi^{-1})^l_k(\phi^{-1})^m_k - \frac{2}{n+1} \delta^i_{(j}(\phi^{-1})^s_{k)}\phi^l _{sa}(\phi^{-1})^a_l, \]
where we write $\phi^i_j = \pfrac{\phi^i}{x^j}\equiv \pfrac{y^i}{x^j}$ and $\phi^{i}_{jk}= \frac{\partial^2 \phi ^i}{\partial x^j\partial x^k}\equiv \frac{\partial^2 y^i}{\partial x^j \partial x^k}$.
Defining $\Pi^{i}_{jk}=-z^i_{jk}$ and similarly for the tilded quantities, and setting $(\phi^{-1})^i_j=\frac{\partial x^i}{\partial y^j}$, we see that the above can be rewritten:
\begin{equation}\label{ThomasChangeOfCoordinates} \tilde{\Pi}^i_{jk}=\pfrac{y^i}{x^l}\frac{\partial^2x^l}{\partial y^j\partial y^k}-\frac{2}{n+1}\pfrac{y^m}{x^l}\frac{\partial^2x^l}{\partial y^m\partial y^{(j}}\delta^i_{k)}+\pfrac{y^i}{x^l}\Pi^{l}_{\,sm}\pfrac{x^s}{y^j}\pfrac{x^m}{y^k}. \end{equation}
Using the definition~\eqref{ThomasInvariant} of the Thomas invariant of a class of projectively equivalent affine connections, it is straightforward to check that its components also under a change of coordinates according to Eq.(\ref{ThomasChangeOfCoordinates}). This justifies the notation $\Pi$ used above. We will thus continue to refer to this object as the Thomas invariant even without reference to an equivalence class of affine connections. 

It will sometimes be convenient to view $\Pi$ as a matrix-valued one form by setting:
\[\Pi^i_j= \Pi^i_{jk}\dd x^k. \]
A section $s$ defining a projective structure determines a reduction of $P^2(\Mbar)$ to a $H$-principal bundle $P$ defined by:
\begin{equation}\label{ProjectiveFrames} P= \{ q \in P^2(\Mbar), q\textrm{~mod~} H = s(\pi(q))\}.\end{equation}
This is again a manifold with boundary. Recall that a Cartan geometry modeled on $G/H$ consists in the following data:
\begin{itemize}
\item A principal $H$-bundle $\pi: P \rightarrow M$,
\item A $\mathfrak{g}$-valued differential form $\omega$ on $P$, known as a \emph{Cartan connection}, that satisfies:
\begin{enumerate}
\item $R_h^*\omega=\Ad_{h^{-1}}\omega$, with respect to the right action of $H$ on $P$.
\item $\omega(X^*)=X$, for any fundamental vector field $X^*$ on $P$ generated by $X\in \mathfrak{h}$.
\item $\omega_p: T_p P \longrightarrow \mathfrak{g} $ is a linear isomorphism at each $p$.
 \end{enumerate}
\end{itemize}
In the above $\mathfrak{g}$ (resp. $\mathfrak{h}$) is the Lie algebra of $G$ (resp. $H$). We recall that the curvature of the Cartan connection is the $\mathfrak{g}$-valued 2-form given by:
\[\Omega=\dd \omega + \frac{1}{2}[\omega,\omega], \]
with $[\omega,\omega](X,Y)=2[\omega(X),\omega(Y)]$. A Cartan geometry is said to be \emph{torsion-free} if its curvature $2$-form $\Omega$ is $\mathfrak{h}$-valued.  
For any local section $\sigma:  U \rightarrow P$ of $\pi: P \rightarrow M$, the pullback form $\omega_U=\sigma^*\omega$ is called a \emph{Cartan gauge} on $U$.
\begin{rema} These definitions are not specific to projective geometry, and apply in the case where $G$ is an arbitrary Lie group of which $H$ is a closed subgroup. The homogeneous space $G/H$ is then referred to as the \emph{model} for the Cartan geometry.\end{rema}
 In the specific case of $G=\PGL(n;\R)$ and $H$ given by~\eqref{StabiliserPointProjectiveGeometry} corresponding to projective geometry, one says that a Cartan geometry is \emph{normal} if it is torsion-free and the curvature function $K: P \rightarrow \Lambda^2(\mathfrak{g}/\mathfrak{h},\mathfrak{h})$ defined by $K(p)(u,v)=\Omega(\omega_p^{-1}u,\omega_p^{-1}v)$, takes its values in the kernel of the so-called projective Ricci morphism defined by the sequence of $H$-module morphisms in Figure~\ref{RicciDef}.
 \begin{figure}[h!]
\begin{center}
\begin{tikzcd} \wedge^2(\mathfrak{g}/\mathfrak{h},\mathfrak{h}) \arrow[r,swap,"\textnormal{proj.}"] & \wedge^2(\mathfrak{g}/\mathfrak{h},\mathfrak{h}/\ker\mathfrak{ad})\arrow[r,"\simeq"]&(\mathfrak{g}/\mathfrak{h})^*\wedge(\mathfrak{g}/\mathfrak{h})^*\otimes \textrm{End}(\mathfrak{g}/\mathfrak{h}) \arrow[d,"\simeq"]\\&(\mathfrak{g}/\mathfrak{h})^* \otimes (\mathfrak{g}/\mathfrak{h})^* & \arrow{l}[yshift=-2mm]{\textrm{contraction in the middle}}(\mathfrak{g}/\mathfrak{h})^*\wedge(\mathfrak{g}/\mathfrak{h})^*\otimes \mathfrak{g}/\mathfrak{h}\otimes(\mathfrak{g}/\mathfrak{h})^*\end{tikzcd}
\end{center}
\caption{\label{RicciDef}The projective Ricci morphism; $\mathfrak{ad}: \mathfrak{h} \rightarrow \textrm{End}(\mathfrak{g}/\mathfrak{h})$ denotes the Lie algebra representation induced by the adjoint action on $\mathfrak{g}/\mathfrak{h}$. }
\end{figure}

The following result extends to our setting: 
\begin{prop}\label{NormalCartanManifoldWithBoundary}
There is a unique normal projective Cartan connection\footnote{The decomposition is with respect to the standard $|1|$-grading:
$\mathfrak{g}=\R^n \oplus \mathfrak{gl}(n,\R) \oplus (\R^n)^*$} $\omega=(\omega^i, \omega^i_j, \omega_j)$ on the bundle $P$ defined in Equation~\eqref{ProjectiveFrames}, such that the components $(\omega^i, \omega^i_j)$ are given by the pullback of the canonical form to $P$.
\end{prop}
\begin{proof}
If such a connection exists, then it is uniquely determined by the induced projective structure on the interior $M$ by density of $M$ in $\Mbar$. To establish existence, let $(U,x)$ be a neighbourhood chart of a point $p$, and $\tilde{s}$ denote the expression of $s$ in this chart. By definition of smoothness the section $\tilde{s}$ determines a section of $P^2(\tilde{U})/H$ where $\tilde{U}$ is an open neighbourhood of $x(U)$ in $\R^n$.  This determines a projective structure on $\tilde{U}$ admitting a unique normal Cartan connection satisfying the requirements of the proposition. Since the set of interior points in $U$ map to a dense set $U$, the restriction of the Cartan connection on $\tilde{U}$ to $x(U)$ is completely determined by the unique normal Cartan connection on $\textrm{int}~U$ and therefore is independent of the choice of extension. This proves local existence, a standard partition of unity arguments gives global existence. 
\end{proof}

{For completeness, we close this subsection with a justification that the standard link between projective structures and equivalence classes of torsion free affine connections carries over without major modification to the case of manifolds with boundary. For the sake of brevity, we shall model our discussion on the exposition given in~\cite[\S7, Chapter IV]{Kobayashi:1995aa} for manifolds without boundary and observe that the main arguments given there extend without essential modification. 

\begin{defi}
A torsion-free affine connection on a manifold with boundary is a section $s: \Mbar \rightarrow P^2(\Mbar)/GL_n(\R)$, where $GL_n(\R)$ is identified with the subgroup $\{(A^i_j, 0)\} \subset G^2(\R^n)$.
\end{defi}
The fact that this definition is a meaningful extension of the usual notion is justified, possibly with minor modifications, by~\cite[Proposition 7.1, Chapter IV]{Kobayashi:1995aa}. As for projective structures, if $(x^i, u^i_j, u^i_{jk})$ are canonical coordinates on $P^2(\Mbar)$ associated with a chart, then we can define a coordinate system on $P^2(\Mbar)/GL_n(\R)$ such that the canonical projection is represented by: 
\[ (x^i, u^i_{j}, u^{i}_{jk}) \mapsto (x^i, u^i_{pq}v^p_j v^q_k),\]
where $u^i_kv^k_j=\delta^i_j$.  Hence, sections of the canonical projection are described by functions $z^i_{jk}$ and one can check  that $\Gamma^i_{jk}=-z^i_{jk}$ transform under change of coordinates like the Christoffel symbols. Every affine connection, in the above sense,  also defines a projective structure in the sense of Definition~\ref{DefinitionProjectiveStructure} and two connections $s,\hat s$  define the same projective structure if and only if their respective Christoffel symbols are locally related by:
\[ \hat{\Gamma}^i_{jk}-\Gamma^i_{jk}=2\delta^i_{(j} \Upsilon_{k)}\]
for some one form $\Upsilon$, which is the local expression of Fact~\ref{FactProjectiveEquivalence}.
Finally, that all projective structures can be seen to arise in this way can be justified by reversing the relation in Eq.~\eqref{ThomasInvariant}, choosing an arbitrary $1$-form for the trace $\alpha^k_k$ of the connection form $\alpha^i_j=\Gamma^i_{jk}\textrm{d}x^k$. We summarise this in the following:
\begin{prop}
Projective structures on manifolds with boundary are in bijective correspondence with equivalence classes of projectively equivalent affine connections.
\end{prop}
}

\subsection{Automorphisms}
\begin{defi}
An automorphism of a projective structure $s: \Mbar \rightarrow P^2(\Mbar)/H$ is a diffeomorphism $\phi: \Mbar \rightarrow \Mbar$ such that its prolongation to $P^2(\Mbar)$ preserves the subbundle $P$ determined by $s$.
\end{defi}
Since the action of the prolongation of $\phi$ to $P^2(\Mbar)$, commutes with the right action of $H$, it follows that one can define a diffeomorphism $\tilde{\phi}$ by means of the commutative diagram below:
\begin{center}
\begin{tikzcd}
P^2(\Mbar) \arrow[d] \arrow[r, "\phi^{(2)}"] & P^2(\Mbar) \arrow[r] & P^2(\Mbar)/H\\
P^2(\Mbar)/H \arrow[rru,dashed, "\tilde{\phi}"] &&
\end{tikzcd}
\end{center}
and we may restate the condition defining an automorphism $\phi$ of $s$ as  
\[ \tilde{\phi}^*s =s. \] 
Locally, this reproduces Lemma~\ref{ThomasProjectiveTransformation}. 
Just as for manifolds without boundary, it follows from the above definitions that a projective transformation of a projective structure on $\Mbar$ is an automorphism of the corresponding normal Cartan geometry given by Proposition~\ref{NormalCartanManifoldWithBoundary}; which we recall defines, by the Cartan connection, an absolute parallelism on $P$.

The following statements are extensions of~\cite[Theorems 3.1,3.2, \S 3, Chapter I]{Kobayashi:1995aa} to the setting of manifolds with boundary.
\begin{prop}\label{Kobayashi31}
Let $G$ be a group of differentiable transformations of a manifold with boundary $\Mbar$, let $S$ be the set of all vector fields $X$ on $\Mbar$ which generate global 1-parameter groups of transformations $\phi_t$ on $\Mbar$ such that $\phi_t\in G$. If the set $S$ generates a finite-dimensional Lie algebra of vector fields on $\Mbar$, then $G$ is a Lie transformation group and $S$ is the Lie algebra of $G$. 
\end{prop}
\begin{proof}
The following observations justify that the proof given in~\cite[\S3, Chapter I]{Kobayashi:1995aa} extends without modification to the present setting. Indeed, we note that $S$ is necessarily composed of vector fields $X$ on $\Mbar$ that are tangent to the boundary $\partial M$, i.e. for every $p \in \partial M$, $X_p \in T_p\partial M \subset T_p \Mbar$. By naturality of the Lie bracket, it follows that this tangency is stable when taking the Lie bracket closure of $S$. This means that the flow theorem can be applied in the form given by~\cite[Theorem 9.34]{Lee:2003aa} to any vector field in the Lie algebra generated by $S$. 
\end{proof}

\begin{prop}
Let $\Mbar$ be a connected manifold with boundary endowed with a $\{1\}$-structure, i.e. an absolute parallelism, then the group of automorphisms $G$ of the $\{1\}$-structure is a Lie transformation group of dimension $\dim G  \leq n= \dim \Mbar$. 
\end{prop}
\begin{proof}
A $\{1\}$-structure amounts to the existence of $n=\dim \Mbar$ everywhere linearly independent vector fields $e_1,\dots,e_n$ on $\Mbar$. Let $\bar{\mathfrak{l}}$ the set of vector fields $X$ that are tangent to $\partial M$ and such that $[X,e_i]=0$ for every $i \in \{1,\dots,n\}$, as the intersection of Lie algebras this is a Lie algebra. 
Since $M$ is dense in $\Mbar$, the restriction map $X \mapsto X|_M$ is injective and maps $\bar{\mathfrak{l}}$ into the Lie algebra $\mathfrak{l}$  of vectors fields on $M$ that commute with each of the vectors fields  $(e_i|_M)$ that define an absolute parallelism of $\Mbar$.
By~\cite[Lemma 3, \S 3 Chapter I]{Kobayashi:1995aa}, $\textrm{dim}~\mathfrak{l} \leq \textrm{dim}~M=\textrm{dim}~\Mbar$, it follows that the same is true of $\bar{\mathfrak{l}}$.
Now let $S$ be the set of vector fields in $\bar{\mathfrak{l}}$ that generate a global $1$-parameter of transformations of $\Mbar$ (recall that the elements of $\bar{\mathfrak{l}}$ are assumed to be tangent to the boundary), by the preceding argument the Lie algebra it generates is therefore of finite dimension, and one can apply Proposition~\ref{Kobayashi31}.
\end{proof}
\begin{rema} 
A $\{1\}$-structure on $\Mbar$ should be thought of as a global section of the first order frame bundle $P^1(\Mbar)$. However, it is possible that no \emph{constant} linear combination of the vector fields composing the global frame is tangent to the boundary at more than one point. %
\end{rema}

\begin{coro}
The group of automorphisms of a projective structure on a manifold with boundary is a Lie group of dimension $\leq n(n+2)=\textnormal{dim}~\PGL(n;\R)$.
\end{coro}

\begin{rema}
We do not expect this upper bound to be sharp; the presence of the boundary will likely impact the dimension of the automorphism group in this case. This has already been observed in case of Riemannian metrics~\cite{Bagaev2003,Chen2010} where the authors show that the automorphism group has at most dimension $\frac{1}{2}n(n-1)$ (as opposed to $\frac{1}{2}n(n+1)$). It should be noted however that the authors consider diffeomorphisms of the canonical \emph{orbifold} structure induced by the manifold with boundary structure; which is a priori a slightly more restrictive notion of diffeomorphism. As an example in the projective case, one may seek all homographies of $n$-dimensional real projective space that restrict to automorphisms of a half-space (see Example~\ref{ExNonRigid}), then one can observe that these are parametrised by matrices of the form:
\[\begin{pmatrix} \lambda & 0 & 0 \\ b & A & w \\ 0 & 0 & a  \end{pmatrix} \mod Z , \quad w,b\in \R^{n-1}, \lambda, a\neq 0, A \in GL_{n-1}(\R), \]
i.e. there are $(n-1)^2+2(n-1)+1=n^2<n(n+2)=\textnormal{dim}~\PGL(n;\R)$ parameters.
\end{rema}
We now move on to a key result:
\begin{prop}\label{PropKobayashi}
Let $G$ be the automorphism group of a $\{1\}$-structure on a connected manifold with boundary $\Mbar$, then for every $p\in \Mbar$ the map: ${\phi}\in G\mapsto {\phi}(p)$ is injective.
\end{prop}
\begin{proof}
{The proof of Proposition~\ref{PropKobayashi} amounts to showing that if $\phi \in G$ has a fixed point, then $\phi$ must be the identity map on $\overline{M}$.}
First let $p\in M$ and suppose ${\phi}(p)=p$ then the restriction of ${\phi}$ to $M$ is an automorphism of the induced $\{1\}$-structure on $M$, hence, ${\phi}|_M = \textrm{id}_M$ by~\cite[Theorem 3.2]{Kobayashi:1995aa}. (Observe that the interior of a connected manifold with boundary is itself connected). By density of $M$ in $\Mbar$ we therefore conclude that ${\phi}=\textrm{id}_{\Mbar}$.

Now let $p\in \partial M$ and assume that ${\phi}(p)=p$; it is sufficient to show that ${\phi}(q)=q$ for at least one interior point. Let $v$ be a constant linear combination of the absolute parallelism $(e_1,\dots, e_n)$ that is inward pointing at $p$. By continuity, it remains inward pointing in a small neighbourhood of $p$. Let $\gamma : [0,\varepsilon) \rightarrow \Mbar$ be an integral curve\footnote{Since we are working locally, we may assume that $v$ is extended to an open neighborhood $\tilde V$ of a manifold without boundary containing the image of a neighborhood $V$ of $p\in \bar M$ through an embedding. On $\tilde V$, we may apply the standard existence and uniqueness theorem for ODEs and note that the restriction of the solution $\gamma$ to $V$ is independent of the extension.} with initial condition $\gamma(0)=p$, and consider the curve ${\phi}(\gamma)$. Then since ${\phi}_*v=v$ by assumption, we have \[\frac{d}{dt} {\phi}(\gamma(t))={\phi}_{*\gamma(t)}\dot \gamma(t)= {\phi}_{*\gamma(t)}v_{\gamma(t)}=v_{{\phi}(\gamma(t))}, \quad {\phi}(\gamma(0))=\gamma(0)=p. \] It follows by uniqueness of integral curves that ${\phi}(\gamma(t))=\gamma(t)$. Since $v$ is inward pointing, there exists at least one fixed point in the interior of $\bar M$.
\end{proof}
\begin{coro}\label{2JetDeterminesPhi}
An automorphism of a projective structure is determined by the $2$-jet at a single point $p\in\Mbar$.
\end{coro}
\begin{proof}
Let $\phi$ be a projective automorphism then its prolongation $\phi^{(2)}$ to $P^2(\Mbar)$ is an automorphism of the $\{1\}$-structure given by the corresponding normal Cartan connection on the subbundle $P$ defined by Equation~\eqref{ProjectiveFrames}. Hence, {by Proposition~\ref{PropKobayashi}}, it is determined on $P$ by the value it takes on a single $2$-frame. In canonical coordinates, one can consider the frame $p$ with coordinates $(x^i,\delta^i_j, \Pi^{i}_{jk})$. The image of this frame under $\phi^{(2)}$ is given by $((\phi(x))^i,\pfrac{\phi^i}{x^j},\frac{\partial^2\phi^i}{\partial x^j \partial x^k}+\pfrac{\phi^i}{x^m}\Pi^{m}_{jk})$ where $x^i$ are the coordinates of an arbitrary point $x\in \Mbar$. Now, $\phi^{(2)}(p)=\tilde {\phi}^{(2)}(p)$ if and only if $\phi$ and $\tilde \phi$ have the same $2$-jet at $x$.
\end{proof}

\section{Proof of the rigidity theorem}
This section is devoted to the proof of Theorem~\ref{ThmRigidity}. We assume that $\Mbar$ is equipped with a fixed  projective structure. %
The strategy of our proof follows closely that of~\cite{WRBLionheart_1997} and relies on Corollary~\ref{2JetDeterminesPhi}. %
Our main objective is to determine at a fixed boundary point $p$ the $2$-jet of a projective transformation $\phi$ that satisfies $\phi|_{\partial M}=\id_{\partial M}$. Since $\phi(p)=p$ we can work in a fixed boundary chart $(U,(r,y))$. Rescaling $r$ by a constant if necessary, we can assume that up to order $2$ in $r$, $\phi$ has the following expression in the chart:
\begin{equation}\label{2jetDiffeo} \begin{cases}\phi^0(r,y)=r+a(y)r^2+o(r^2), \\ \phi^\mu(r,y) = y^\mu + b^{\mu}(y)r+c^{\mu}(y)r^2 + o(r^2).\end{cases} \end{equation}
The goal is to show that $a,b^\mu,c^\mu$ vanish at $p$. 

Observe that in light of Equation~\eqref{2jetDiffeo} at boundary points we have:
\begin{equation}\label{gAtBoundaryPoint} \left(\pfrac{\phi^i}{x^j}\right)=\begin{pmatrix}1 & 0 \\ b^\alpha & \delta^\alpha_\beta \end{pmatrix}, \quad \left(\pfrac{\phi^i}{x^j}\right)^{-1}\equiv\left(\pfrac{x^i}{\phi^i}\right)=\begin{pmatrix}1 & 0 \\ -b^\alpha & \delta^\alpha_\beta \end{pmatrix}\end{equation}

\begin{prop}\label{PropEquationsABC}
The functions $a, b^\mu, c^\mu$ satisfy the following system:
\begin{equation}\label{SystemGeneral}\begin{cases}b^\mu\Pi^0_{\,\nu \tau}=0, \\ \partial_\nu b^\mu -\frac{\delta^\mu_\nu}{n+1}(2a+\partial_\tau b^\tau)=b^\mu\Pi^0_{0\nu}-\Pi^\mu_{\nu\tau}b^\tau,\\ -2b^\mu a +2c^\mu - b^\mu\Pi^0_{00} - 2b^\mu\Pi^0_{0\nu}b^\nu + 2 \Pi^\mu_{\nu 0}b^\nu + \Pi^{\mu}_{\sigma \nu}b^\sigma b^\nu =0. \end{cases} \end{equation}
\end{prop}
Before we proceed to prove these relations, it follows immediately that:
\begin{coro}
\label{VanishingCoefficients}
If there is a boundary point such that at that point $\Pi^0_\nu(p) \notin \textnormal{span}(\dd r(p))$ for at least one $\nu$ then, at that point:
\[a = b^\mu = c^\mu =0 \]
\end{coro}
Theorem~\ref{ThmRigidity} follows:
\begin{proof}[Proof of Theorem~\ref{ThmRigidity}]
It follows from Corollary~\ref{VanishingCoefficients} that at the point $p\in \partial M$ satisfying the hypotheses of the Theorem, the $2$-jet of the diffeomorphism coincides with that of the identity map. By Corollary~\ref{2JetDeterminesPhi}, it follows that it must be the identity map.
\end{proof}
The rest of the section is devoted to the calculations that prove Proposition~\ref{PropEquationsABC}. First, under a diffeomorphism, $\Pi^i_{jk}$ transforms according to:
\begin{equation*}\tilde{\Pi}^i_{jk}(x)= \pfrac{x^i}{\phi^l}\frac{\partial^2\phi^l}{\partial x^j \partial x^k} -\frac{2}{n+1} \pfrac{x^m}{\phi^l}\frac{\partial^2\phi^l}{\partial x^m \partial x^{(j}}\delta^i_{k)}+\pfrac{x^i}{\phi^l} \Pi^{l}_{sm}(\phi(x)) \pfrac{\phi^s}{x^j}\pfrac{\phi^m}{x^k} \end{equation*}
Using Lemma~\ref{ThomasProjectiveTransformation} and the assumption that $\phi$ restricts to the identity at boundary points, it follows that at any \emph{boundary point}:
\begin{equation}\label{ProjectiveTransfComp}\pfrac{x^i}{\phi^l}\frac{\partial^2\phi^l}{\partial x^j \partial x^k} -\frac{2}{n+1} \pfrac{x^m}{\phi^l}\frac{\partial^2\phi^l}{\partial x^m \partial x^{(j}}\delta^i_{k)}+\pfrac{x^i}{\phi^l} \Pi^{l}_{sm} \pfrac{\phi^s}{x^j}\pfrac{\phi^m}{x^k}=\Pi^i_{jk}. \end{equation}
We must now study what Equation~\eqref{ThomasChangeOfCoordinates} means for the different values of $i,j,k$.
We recall that $\Pi^i_{jk}$ is symmetric in the lower indices and trace-free.
\subsection*{Cases where $i=0$}
Let us consider the case $j=k=0$, then~\eqref{ProjectiveTransfComp} when evaluated at the boundary point $p$:
\[2a -\frac{2}{n+1}\left(2a +\partial_\alpha b^\alpha \right)+\Pi^0_{00} + 2\Pi^0_{0\mu}b^\mu + \Pi^0_{\mu\nu}b^\mu b^\nu =\Pi^0_{00}, \]
this can be rewritten:
\begin{equation}\label{Eq1T} \partial_\alpha b^\alpha = (n-1)a +(n+1)\Pi^0_{0\mu}b^\mu +\frac{(n+1)}{2}\Pi^0_{\mu\nu}b^\mu b^\nu. \end{equation}
In the cases $\{j,k\}=\{0,\mu\}$ the left-hand side of~\eqref{ProjectiveTransfComp} expands to:
\[\frac{\partial^2 \phi^0}{\partial x^\mu \partial x^0} - \frac{2}{n+2}\left(\frac{1}{2}\frac{\partial^2\phi^0}{\partial x^0 \partial x^\mu}-\frac{1}{2}b^\nu \frac{\partial^2\phi^0}{\partial x^\nu \partial x^\mu} + \frac{1}{2}\frac{\partial^2\phi^\nu}{\partial x^\nu \partial x^\mu}\right)+\Pi^{0}_{\mu 0}+ \Pi^{0}_{\mu \nu}b^\mu. \]
The first two terms vanish at boundary points and it follows that:
\begin{equation}\label{Eq2T} \Pi^0_{\mu \nu}b^\nu =0. \end{equation}
The case $i=\mu, j=\nu$ reduces to a trivial equation.
\subsection*{Cases where $i=\mu$}
When $j=k=0$ we obtain:
\[-2 a b^\mu + 2c^\mu -b^\mu\left( \Pi^0_{00} + 2\Pi^0_{0\nu}b^\nu + \Pi^0_{\mu\nu}b^\mu b^\nu\right) + \Pi^{\mu}_{00} + 2 \Pi^{\mu}_{0 \nu}b^\nu +\Pi^{\mu}_{\sigma\nu}b^\sigma b^\nu=\Pi^{\mu}_{00}.\]
Using~\eqref{Eq2T} this affords a minor simplification and yields:
\begin{equation}\label{Eq3T} 2c^\mu - 2b^\mu a= b^\mu\Pi^{0}_{00} +2b^\mu \Pi^{0}_{0\nu}b^\nu -2\Pi^{\mu}_{\nu 0}b^\nu -\Pi^{\mu}_{\sigma \nu}b^\sigma b^\nu. \end{equation}
The cases $\{j,k\}=\{\nu,0\}$ give:
\[ \partial_\nu b^\mu -\frac{2 \delta^\mu_\nu}{n+1}\left( a +\frac{1}{2} \partial_\sigma b^\sigma \right)-b^\mu\Pi^0_{\nu0} -b^\mu \Pi^0_{\nu\sigma}b^\sigma +\Pi^\mu_{\nu0} + \Pi^\mu_{\nu \sigma}b^\sigma=\Pi^{\mu}_{\nu0}, \]
hence:
\begin{equation}\label{Eq4T} \partial_\nu b^\mu = \frac{\delta^\mu_\nu}{n+1}\left(2a +\partial_\tau b^\tau\right)+b^\mu\Pi^0_{\nu0}- \Pi^\mu_{\nu\sigma}b^\sigma.\end{equation}
Using that: $0=\Pi^i_{i\sigma}= \Pi^0_{0\sigma} + \Pi^{\mu}_{\mu \sigma}$, one may observe that Eq.~\eqref{Eq1T} is simply the trace of this equation and may be discarded.
Finally, what turns out to be the key case, $j=\nu, k=\tau$:
\[\begin{split}-b^\mu\frac{\partial \phi^0}{\partial x^\nu \partial x^\tau} + \frac{\partial^2 \phi^\mu}{\partial x^\nu \partial x^\tau}-\frac{2}{n+1}\left( \frac{\partial^2\phi^0}{\partial x^0 \partial x^{(\nu}}\delta^\mu_{\tau)}-b^\lambda \frac{\partial^2\phi^0}{\partial x^\lambda \partial x^{(\nu}}\delta^\mu_{\tau)} + \frac{\partial^2 x^\lambda}{\partial x^\lambda \partial x^{(\nu}}\delta^\nu_{\tau)} \right) \\ -b^\mu\Pi^0_{\nu\tau} +\Pi^\mu_{\nu\tau}=\Pi^\mu_{\nu\tau}. \end{split} \]
At boundary points, all the terms on the first line vanish and it follows that:
\begin{equation}\label{Eq5T} b^\mu \Pi^0_{\nu\tau}=0. \end{equation}
Once more, we can observe that~\eqref{Eq2T} is the trace of~\eqref{Eq5T} and may be discarded so that the independent equations are Eqs.~\eqref{Eq3T},\eqref{Eq4T} and ~\eqref{Eq5T} which leads to the desired system.\qed
\section{Non-rigidity}\label{NonRigidity}
\subsection{First steps}
Equation~\eqref{Eq5T} is quite powerful since, as we have already remarked, when any of the forms $\Pi^0_{\nu}$ are non-vanishing the whole system collapses and is immediately solved. For our purposes, we would only need to find one boundary point and one boundary chart at which this holds and the problem is solved. We first show that the condition is chart independent.

\begin{lemm}\label{CoordinateIndependanceLemma}
Let $p_0\in \partial M$. Suppose that $\Pi^0_{\nu\mu}$ vanishes at $p_0$ for every $\nu,\mu$ in some chart, then it must vanish in every chart. 
\end{lemm}

\begin{proof} Let the hypotheses of the Lemma be satisfied in some chart $(x,U)$ and let $(\bar{x},V)$ be another boundary chart near $p_0$. Let us write, $\bar{x}^0=\bar{r}$ and $\bar{x}^\mu=\bar{y}^\mu$. Note that $\bar{r}$ is a boundary defining function for the boundary on $V$ and, {\it a fortiori}, on $U\cap V$. Recall that the Thomas invariant transforms under change of coordinates according to
\[ \bar{\Pi}^i_{\, jk}= \pfrac{\bar{x}^i}{x^l}\frac{\partial^2x^l}{\partial \bar{x}^j\partial \bar{x}^k}-\frac{2}{n+1}\pfrac{\bar{x}^m}{x^l}\frac{\partial^2x^l}{\partial\bar{x}^m\partial\bar{x}^{(j}}\delta^i_{k)}+\pfrac{\bar{x}^i}{x^l}\Pi^{l}_{\,sm}\pfrac{x^s}{\bar{x}^j}\pfrac{x^m}{\bar{x}^k}. \]
We are interested in the case $i=0, j=\mu, k=\nu$; the middle term vanishes because of the Kronecker delta, then the rest can be expanded:
\[\begin{aligned}
\bar{\Pi}^0_{\, \mu\nu}&= \pfrac{\bar{x}^0}{x^l}\frac{\partial^2x^l}{\partial \bar{x}^\mu\partial \bar{x}^\nu}+\pfrac{\bar{x}^0}{x^l}\Pi^{l}_{\,sm}\pfrac{x^s}{\bar{x}^\mu}\pfrac{x^m}{\bar{x}^\nu},\\&=\pfrac{\bar{r}}{r}\frac{\partial^2r}{\partial\bar{y}^\mu\partial\bar{y}^\nu}+\pfrac{\bar{r}}{y^\beta}\frac{\partial^2 y^\beta}{\partial\bar{y}^\mu\partial\bar{y}^\nu} + \pfrac{\bar{r}}{r}\left(\Pi^0_{\,00}\pfrac{r}{\bar{y}^\mu}\pfrac{r}{\bar{y}^\nu}+\Pi^0_{\,0\beta}\pfrac{r}{\bar{y}^\mu}\pfrac{y^\beta}{\bar{y}^\nu}+\Pi^0_{\,\beta 0}\pfrac{y^\beta}{\bar{y}^\mu}\pfrac{r}{\bar{y}^\nu} \right)\\&\hspace{1cm}+\pfrac{\bar{r}}{r}\Pi^0_{\,\beta\sigma}\pfrac{y^\beta}{\bar{y}^\mu}\pfrac{y^\sigma}{\bar{y}^\nu}+\pfrac{\bar{r}}{y^\beta}\Pi^\beta_{\,sm}\pfrac{x^s}{\bar{x}^j}\pfrac{x^m}{\bar{x}^k}.
\end{aligned}
 \]
Now since both $r$ and $\bar{r}$ are boundary defining functions on $U\cap V$ there is a non-vanishing function $y\mapsto c(y)$ such that:
\[ \dd \bar{r} = c \dd r,\]
holds at boundary points.
This implies that $\displaystyle \lim_{r\to 0} \pfrac{\bar{r}}{y^\mu}=0$ and $\displaystyle \lim_{\bar{r}\to 0} \pfrac{r}{\bar{y}^\mu}=0.$
Hence taking the limit in the above equation we get, at boundary points:

\[\bar{\Pi}^0_{\mu\nu} = c \Pi^0_{\beta \sigma} \pfrac{y^\beta}{\bar{y}^\mu}\pfrac{y^\sigma}{\bar{y}^\nu}.\]
Therefore, if $\Pi^0_{\beta\sigma}$ vanishes for every $\beta,\sigma$ the same holds for $\bar{\Pi}^0_{\mu\nu}.$
\end{proof}

It is apparent that we have no choice but to consider the other equations of~\eqref{SystemGeneral} in this case. To get an idea of what is going on, let us see if we can solve them and prove rigidity in the extreme case where $\Pi$ vanishes at every boundary point. The system becomes very manageable:
\[\begin{cases}\partial_\nu b^\mu = \delta^\mu_\nu a,\\c^\mu=a b^\mu,
 \end{cases} \]
 and is readily solved to yield \emph{non-trivial} solutions parametrised by a constant $a \in \R$
 \[ \begin{cases} b^\mu = a y^\mu +k^\mu, \\ c^\mu= a b^\mu. \end{cases} \]

The following example shows a simple case in which some of these solutions can be achieved and that our search for rigidity must end here.
\begin{exam}\label{ExNonRigid} 
Let us consider an example in the projective plane $\RP^2$.
Take the boundary to be defined in homogenous coordinates $[X,Y,Z]$ by $\{XZ =0 \}$, in the affine chart sending $\{Z=0\}$ to infinity, we have a line and we can define a manifold with boundary by sitting on one side, say $x>0$ in the affine coordinates $(x,y)$ of this affine part of the projective plane. The flat connection is well defined on the boundary and vanishes there like everywhere else.
The projective transformations defined by matrices of the form
\[ M = \begin{pmatrix} 1 & 0 & 0 \\ \beta & 1& 0 \\ 0 & 0 & 1 \end{pmatrix},\]
are projective transformations that are tangent to the identity at the boundary but are not the identity map. 
They are the only %
projective transformations that can be obtained by restriction of elements of the projective group to our manifold.
In affine coordinates:
\[ (x,y) \overset{\phi}\longmapsto \left(x, \beta x + y\right),\]
\end{exam}

This example motivates us to give further attention to the geometry of the non-rigid boundary, which is the object of the next section of our paper.
\subsection{Cartan geometry at the boundary in the case of non-rigidity.}\label{ssec:bCartanGeometry}
{Recall from \cite[Chapter 5; Definition 1.2, Proposition 2.5]{Sharpe:1997aa} (see also the discussion on Cartan geometries in Section~\ref{SectionDefinitionProjectiveStructure}) that if $\sigma : U \subset \Mbar \rightarrow P$ is a local section of a Cartan geometry $(P,\omega)$, then the pullback of the Cartan connection, $\sigma^*\omega\equiv\omega_U$ is referred to as a Cartan gauge.} The normal Cartan geometry associated with a class of projectively equivalent {torsion-free affine} connections admits Cartan gauges of the form:
\begin{equation}\label{LocalNormalCartanConnection}\omega_U= \begin{pmatrix}\alpha_0^0-\frac{1}{n+1}\alpha^l_l & \alpha^0_\mu & \dd r \\ 
\alpha^\mu_0 & \alpha^\mu_\nu -\frac{1}{n+1}\alpha^l_l\delta^\mu_\nu & \dd y^\mu\\ - P_{i0}\dd x^i & -P_{i\nu}\dd x^i & -\frac{1}{n+1}\alpha^l_l\end{pmatrix},\end{equation}
where $\alpha^i_j$ are the local connection forms on $U$ of an affine connection $\nabla$ in the projective class, and 
$P_{ij}$ are the components of the (projective) Schouten tensor of the given connection defined by (in abstract index notation):
\[(n-1)P_{ab}= R_{ab}-\frac{2}{n+1}R_{[ab]}.\]
In the above $R_{ab}$ is the Ricci tensor $R_{ab}=R_{ca\phantom{c}b}^{\phantom{ca}c}$. This is not a projectively invariant quantity as it depends on the choice of connection and varies under a projective change of connection by:
\[\hat{P}_{ab}=P_{ab}-\nabla_a\Upsilon_b +\Upsilon_a\Upsilon_b. \]
If one performs a change of gauge described by the matrix\footnote{{We mean here that we calculate the Cartan gauge $\tilde \omega_U$ given by pulling back the Cartan connection along the local section $\tilde{\sigma}=\sigma h$ where the original gauge is defined by the section $\sigma: U \rightarrow P$. The new gauge $\tilde\omega_U$ is obtained from $\omega_U$ by the standard formula given in Eq.~\eqref{ChangeOfTrivialisation}. }}:
\[h=\begin{pmatrix} I_n & 0 \\ -\frac{1}{n+1}\alpha^l_{lj} & 1 \end{pmatrix}, \]
then we can obtain an expression of the normal Cartan connection in terms of the Thomas invariant:
\begin{equation}\label{LocalGaugeWithThomas} \begin{pmatrix} \Pi^0_0 & \Pi^0_\mu & \dd r \\ \Pi^\mu_0 & \Pi^\mu_\nu & \dd y^\mu \\ -\tilde{P}_{k0}\dd x^k & -\tilde{P}_{k\nu}\dd x^k & 0  \end{pmatrix}, \end{equation}
If we set $\chi=-\frac{1}{n+1}\alpha^l_l$, then $\tilde{P}_{ab}=P_{ab}-\nabla_a\chi_b+\chi_a\chi_b$. It is easily checked that $\tilde{P}$ is projectively invariant, {i.e. only depends on $\Pi^i_{jk}$, see in particular Eq.~\eqref{AnalogueSchoutenThomas}}
\begin{rema}
It is interesting to observe that in order to preserve the above form for the Cartan gauge under a change of coordinates $x\to \bar{x}$, the corresponding change of gauge is given by:
\begin{equation}\label{ChangeOfGauge} \begin{pmatrix} \pfrac{x^i}{\bar{x}_j} & 0 \\ -\frac{1}{n+1}\pfrac{\bar{x}^l}{x^k}\frac{\partial^2 x^k}{\partial\bar{x}^l\partial\bar{x}^j} & 1  \end{pmatrix} \mod Z.\end{equation}
This correctly implements Equation~\eqref{ThomasChangeOfCoordinates} on the first $n\times n$ block of the local connection form. One can think of choosing the Cartan gauge of the connection to be given by~\eqref{LocalGaugeWithThomas} as being equivalent to the choice of coordinates on $P^2(\Mbar)/H$ we made at the beginning of Section~\ref{SectionDefinitionProjectiveStructure}.
\end{rema}
\begin{rema}
If one considers the matrix-valued $2$-form:
\[\tilde{\Omega}^i_j = \dd \Pi^i_j + \Pi^i_k\wedge \Pi^k_j,  \] then we have the following expression of $\tilde{P}_{ab}$ defined above:
\begin{equation}\label{AnalogueSchoutenThomas} \tilde{P}_{ab}= \frac{1}{n-1}\tilde{\Omega}^i_j\left(\pfrac{}{x^i},\pfrac{}{x^l}\right)\dd x^l_{a}\dd x^j_{b}. \end{equation}
One may observe that since $\Pi^i_i=0$, and hence $\tilde{\Omega}^i_i=0$, $\tilde{P}$ is symmetric.
\end{rema}
Equation~\eqref{LocalGaugeWithThomas} gives the expression of the unique normal Cartan connection determined by a projective structure on $\Mbar$ (see Proposition~\ref{NormalCartanManifoldWithBoundary}), in an appropriate gauge.
We are interested in how this normal Cartan connection of a projective structure pulls back to the boundary under the necessary condition for non-rigidity given by $i^*\Pi^0_\mu =0$ for all $\mu$, where $i : \partial M \hookrightarrow \Mbar $ is the canonical inclusion. It follows that:
\begin{equation}\label{PullbackLocalForm}i^*\omega_U=\begin{pmatrix}i^*\Pi_0^0& 0 & 0 \\ 
i^*\Pi^\mu_0 & i^*\Pi^\mu_\nu & \dd y^\mu\\ - \tilde{P}_{\sigma0}\dd y^\sigma & -\tilde{P}_{\sigma\nu}\dd y^\sigma & 0\end{pmatrix}.\end{equation}
One might observe that the Cartan gauges take their values in a Lie subalgebra of $\mathfrak{g}$, that is, the Lie algebra of the subgroup $\IsoHyperplane$ defined by:
\[\IsoHyperplane =\left\{ \begin{pmatrix}a&0&0\\b&A&u\\\beta & \delta &\gamma\end{pmatrix} \mod Z, \quad a\neq 0, \begin{pmatrix}A & u \\ \delta& \gamma \end{pmatrix}\in GL(n;\R) \right\}.\]
This is also exactly the isotropy subgroup of the projective hyperplane in $\RP^n$ defined, in homogenous coordinates, by $\mathcal{H}=\{X_0=0\}$.

We are therefore lead to conjecture that, in this specific case, we are given a canonical Cartan geometry on the boundary of which these Cartan gauges furnish the local data. It remains to identify an appropriate closed subgroup $\tilde{H}$ of $\tilde{G}$ such that $\dim\IsoHyperplane - \dim \tilde{H}= n-1$, and an appropriate geometric picture for the model. 

Since $\IsoHyperplane$ stabilises the hyperplane, it naturally acts on it. The action is easily seen to be transitive and hence we can consider the isotropy group of any point on the plane, for instance: ${}^t\begin{pmatrix}0 & 0_{\R^{n-1}}&1\end{pmatrix}$, which is readily identified as :
\[\tilde{H} =\left\{ \begin{pmatrix}a&0&0\\b&A&0\\\beta & \delta &\gamma\end{pmatrix} \mod Z, \quad a,\gamma \neq 0, b\in \R^{n-1}, \beta \in \R, A\in GL(n-1;\R) \right\} \subset H.\]

The proof of Lemma~\ref{CoordinateIndependanceLemma} also provides evidence that this is the correct choice, since coordinate changes between boundary charts of $\Mbar$ naturally restrict to elements $\tilde{H}$ at boundary points. This also implies that there is a natural reduction $Q \hookrightarrow i^*P$ to a principal $\tilde{H}$-bundle.

Observing finally that $i^*\omega_U\mod \tilde{\mathfrak{h}}$ is a linear isomorphism between $T_x\partial M$ and $\tilde{\mathfrak{g}}/\tilde{\mathfrak{h}}$ it follows that:
\begin{prop}
Suppose that at every boundary point, $\Pi^0_\mu \in \vspan (\dd r)$, for some, and hence all, local boundary charts and connections in the projective class. Let $Q \hookrightarrow i^*P$ be the natural reduction to a $\tilde{H}$-principal bundle, obtained by restricting to boundary compatible change of coordinates on $\Mbar$ to $\partial M$. Then pulling back the normal Cartan projective connection to $Q$ leads to a Cartan geometry modeled on the \emph{non-effective} Klein geometry $\IsoHyperplane/\tilde{H}$ on $\partial M$.
 \end{prop}
As mentioned in the previous statement, it should be observed that the action of $\tilde{G}$ on $\mathcal{H}$ is \emph{not} effective, indeed the kernel $K$ is given by:
\[K=\left\{ \begin{pmatrix}a & 0 \\ b & \lambda I_{n}\end{pmatrix} \mod Z, a,\lambda \neq 0, b\in \R^n\right\}. \]
However, the effective geometry associated to this model is projective geometry in $n-1$ dimensions, i.e.
 \[ \tilde{G}/\tilde{H}\cong (\tilde{G}/K) / (\tilde{H}/K) \cong \RP^{n-1}. \]
We shall now show that \enquote{factoring out} the kernel from the Cartan gauges of the non-effective Cartan geometry leads to a Cartan projective structure on $\partial M$. Since projective geometry is an \emph{effective} Klein geometry~\cite[Chapter 5\S2]{Sharpe:1997aa}, a projective Cartan geometry on a manifold without boundary $N$ is uniquely determined by a \emph{Cartan atlas}: a collection: $\{U_i, \omega_{U_i}\}$ where $\{U_i\}$ is an open cover of $N$, and $\omega_{U_i}$ a Cartan gauge on $U_i$ (see Section~\ref{Projective_Structures_On_MWB}). Note that the Cartan gauges satisfy the following compatibility condition: there exists a smooth map $h_{ij}: U_i \cap U_j \rightarrow H$ such that
\begin{equation} \label{ChangeOfTrivialisation} \omega_{U_j}= h_{ij}^*\omega_{H} + \textrm{Ad}_{h_{ij}^{-1}}\omega_{U_i}, \end{equation}
where $\omega_H$ denotes the Maurer-Cartan form of $H$. Therefore we will reason with the Cartan gauges and appeal to~\cite[Chapter 5\S2]{Sharpe:1997aa}.
\begin{prop}
Let $\{U_i\}$ be an open cover of $\partial M$ by boundary charts of $\Mbar$. For each $U_i$, let $\omega_{U_i}$ be the Cartan gauge of the normal Cartan connection on $\Mbar$ given by Eq.~\eqref{LocalGaugeWithThomas}. %

Define $\tilde{\omega}_i=i^*\omega_{U_i} \mod \mathfrak{k}$, where $\mathfrak{k}$ is the Lie algebra of $K$, then $\{U_i\cap \partial M, \tilde{\omega}_i\}$ is a Cartan atlas for a projective geometry on $\partial M$.
\end{prop}
\begin{proof}
We only need to show that the gauges $\tilde{\omega}_i$ are compatible, that is, that there are maps $\tilde{h}_{ij}: U_i \cap U_j \cap \partial M \rightarrow \tilde{H}/{K}$ such that:
\[ \tilde{\omega}_j = \tilde{h}_{ij}^*\omega_{\tilde{H}/K} + \textrm{Ad}_{\tilde{h}_{ij}^{-1}}\tilde{\omega}_i. \]
By definition, there is a map $h_{ij}: U_i\cap U_j \rightarrow H$ given by Equation~\eqref{ChangeOfGauge} that, as remarked above, takes its values in $\tilde{H}$ at boundary points, and such that:
\begin{equation} \label{CompatibilityInTildeH} \omega_{U_j}=h^*_{ij}\omega_{H} + \textrm{Ad}_{h_{ij}^{-1}}\omega_{U_i}. \end{equation}
Set $\tilde{h}_{ij}=\pi \circ h_{ij}$, where $\pi : \tilde{H}\rightarrow \tilde{H}/K$ is the canonical projection morphism. Now, fix $h\in H$ and note that we have the following commutative diagram:
\begin{center}
\begin{tikzcd}
\tilde{H} \arrow[d,"\pi"] \arrow[r,"L_{h^{-1}}"] & \tilde{H} \arrow[r,"\pi"] & \tilde{H}/K\\
\tilde{H}/K  \arrow[swap,urr,"L_{\pi(h^{-1})}"]
\end{tikzcd}
\end{center}
where $L_{\bullet}$ denotes left multiplication in the corresponding group. Differentiating this at $h$, gives the following relationship between the Maurer-Cartan forms:
\[ \pi_{*e}\circ \underbrace{{L_{h^{-1}}}_{*h}}_{(\omega_{\tilde{H}})_{h}} = \underbrace{{L_{\pi(h^{-1})}}_{*\pi(h)}}_{(\omega_{\tilde{H}/\tilde{K}})_{\pi(h)}} \circ \pi_{*h}. \] 
It follows from this that for all $p\in U_i\cap U_j \cap \partial M$, $X\in T_p \partial M$
\[ \begin{aligned}(\pi \circ h_{ij})^*(\omega_{\tilde{H}/\tilde{K}})_p(X)&= (\omega_{\tilde{H}/\tilde{K}})_{\pi(h(p))}\left(\pi_{*h(p)}h_{*p}(X)\right)\\&=\pi_{*e}\left((\omega_{\tilde{H}})_{h_{ij}(p)}{h_{ij}}_*(X)\right)\\&=\pi_{*e} (h^*(\omega_{\tilde{H}})_p(X)). \end{aligned}\]
This identity shows that considering Eq.~\eqref{CompatibilityInTildeH} modulo $\mathfrak{k}$, leads to the compatibility of the gauges $\tilde{\omega}_i$ and $\tilde{\omega}_j$. 
\end{proof}

We have therefore obtained a projective structure on the boundary $\partial M$. 
\subsection{Non-rigidity and geodesics}
The definition of the (unparametrised) geodesics of a projective structure given in~\cite[\S 3, Chapter 8]{Sharpe:1997aa} carries over to manifolds with boundary without essential modification. In terms of the Thomas invariant, using~\cite[Proposition 3.5, Chapter 8]{Sharpe:1997aa}\footnote{One might think that there is a first order term, namely $\varepsilon(\dot{c})\theta^i(\dot{c})$, missing from the numerator in the equation in~\cite[Proposition 3.5, Chapter 8]{Sharpe:1997aa}, however since $\varepsilon(\dot{c})$  does not depend on $i$, it can be included in the function $\lambda$.}, they are the curves $\gamma$ for which there is a function $\lambda$ such that:
\begin{equation}\label{GeodesicEquation} \ddot{\gamma}^i + \Pi^i_{jk}\dot\gamma^j \dot{\gamma}^k = \lambda\dot \gamma^i\end{equation}
\begin{rema} By assumption, in a boundary chart, the functions $\Pi^i_{jk}$ admit extensions to an open set in $\R^n$. The above equations can consequently be extended and studied as equations on an open subset $\R^n$. The usual uniqueness results show the restrictions of solutions to the coordinate neighbourhood of the boundary point are independent of the choice of extension. However, geodesics can now be defined on any type of interval, as they may start or end at boundary points.
\end{rema}

In this section we will study the link between the geodesics of the projective structure on $\Mbar$ to those on $\partial M$. Observing that the local Thomas invariants $\tilde{\Pi}^i_{jk}$ of this induced structure are related to $\Pi^i_{jk}$ at boundary points by the following formula :
\[ \tilde{\Pi}^\mu_{\nu\sigma}=\Pi^\mu_{\nu\sigma} +\frac{2}{n}\Pi^0_{0(\nu}\delta^\mu_{\sigma)}, \]

\begin{prop}
Suppose that at every boundary point, $\Pi^0_\mu \in \vspan (\dd r)$, for some, and hence all, local boundary charts. Then the geodesics of the induced projective structure on $\partial M$ are geodesics of $\Mbar$. In particular, $\partial M$ is totally geodesic. 
\end{prop}
\begin{proof}
Let $\gamma$ be a solution of the corresponding geodesic equation on the boundary, i.e. for every $\mu$:
\[ \ddot{\gamma}^\mu + \tilde{\Pi}^\mu_{\nu\sigma}\dot{\gamma}^\nu\dot{\gamma}^\sigma=\lambda\dot\gamma^\mu \]
then:
\[ \ddot{\gamma}^\mu +\Pi^\mu_{\nu\sigma}\dot{\gamma}^\nu\dot{\gamma}^\sigma=(\lambda-\frac{2}{n}\Pi^0_{0\nu}\gamma^\nu)\dot\gamma^\mu=\tilde{\lambda}\dot\gamma^\mu. \]
Hence, viewing this as a curve on $\Mbar$, we see that this is a solution to Equation~\eqref{GeodesicEquation}, by uniqueness of the solution (up to reparametrisation), this is therefore a geodesic on $\Mbar$.
\end{proof}

We shall conclude this section by showing that this is the full geometric interpretation of non-rigidity condition.
\begin{prop}
Suppose that every geodesic of the projective structure on $\Mbar$ that is initially tangent to $\partial M$ remains on $\partial M$. Then at every boundary point $p\in \partial M$, in an arbitrary boundary chart $(r,y^\mu)$ near $p$, the condition $\Pi^0_\mu(p) \in \textnormal{span}( \dd r_p)$ is satisfied. 
\end{prop}
\begin{proof}
Choose a point $p\in \partial M$ and $X_p \in T_p\partial M$. Considering the geodesic with initial conditions $(p,X_p)$ then by hypothesis for all $t\in I$:
\[ \dd r(\dot{\gamma}(t))=\dot{\gamma}^0(t)=0,\]
where $I$ is a domain of definition of the geodesic.
Using Equation~\eqref{GeodesicEquation} it follows that:
\[\Pi^0_{\mu\nu}(\gamma(t))\dot\gamma^\mu(t)\dot\gamma^\nu (t) =0, \]
for all $t\in I$, evaluating this at $t=0$ and using that $(p,X_p)$ is arbitrary, we conclude that $\Pi^0_{\mu\nu}\equiv 0$ along the boundary.
\end{proof}
\begin{rema}
The gauge in the Cartan atlas of the Cartan geometry produced by the procedure in the proof of Proposition 4.2 is given by:
\[ \begin{pmatrix} i^* \Pi^\mu_\nu + \frac{1}{n}i^*\Pi^0_0 \delta^\mu_\nu & \dd y^\mu \\ -\tilde{P}_{\sigma\nu}\dd y^\sigma & \frac{1}{n}i^*\Pi^0_0 \end{pmatrix}. \]
One might naturally wonder if this is a normal projective Cartan geometry. Translating the diagram in Figure~\ref{RicciDef}, one can see that this is the case if and only if: 
\[ i^*\tilde{\Omega}^\mu_\nu\left(\cdot,\pfrac{}{y^\mu}\right)-\tilde{P}_{\sigma\nu}(\dd y^\mu \wedge \dd y^\sigma)(\cdot, \pfrac{}{y^\mu}) + \tilde{P}_{\sigma\lambda}(\dd y^\sigma \wedge \dd y^\lambda) (\cdot, \pfrac{}{y^\nu})=0, \]
or, after expanding and using the symmetry of $\tilde{P}$:
\[(n-2)\tilde{P}_{\sigma \nu}dy^\sigma= - i^*\tilde{\Omega}^\mu_\nu(\cdot,\pfrac{}{y^\mu}).  \]
This can then be shown to lead to the condition
\[ \tilde{P}_{\sigma\nu}=\tilde{\Omega}^0_{\nu}\left(\pfrac{}{r},\pfrac{}{y^\sigma}\right),\]
that will not be satisfied in general.
\end{rema}

\section{Further examples}
\begin{exam}\label{ExRigid}
 A projective disk (the \enquote{inside} of a projective conic defined by the homogenous polynomial $X^2+Y^2=Z^2$) is (projective) boundary rigid.
Indeed, let us introduce the affine chart which puts the projective line defined by the plane $\{Z=0\}$ at infinity and defining local coordinates $(r,t)$ in the disk by
\[x=(1-r)\cos t, y=(1-r)\sin t.\]
(So that $r=0$ is the boundary).

The Thomas invariant corresponding to the canonical flat projective structure obtained from the projective plane by restriction to the disk is given in these coordinates in matrix form by:
\[ \Pi =\begin{pmatrix}\frac{2}{3}\frac{\dd r}{1-r} & (1-r)\dd t \\ -\frac{2}{3}\frac{\dd t}{1-r} & -\frac{2}{3}\frac{\dd r}{1-r} \end{pmatrix}\] so at a boundary point:
\[ \Pi = \begin{pmatrix} \frac{2}{3}\dd r & \dd t \\ -\frac{2}{3} \dd t & -\frac{2}{3}\dd r \end{pmatrix}.\]
Which satisfies the assumptions of our Theorem~\ref{ThmRigidity}.

\end{exam}
To put our theorem to the test, let us search for diffeomorphisms which restrict to the identity on the boundary of the disk, within the connected component of the identity of $O(2,1)\mod Z$. These can be parametrised by Euler \enquote{angles} as follows:
\[g=\begin{pmatrix} \cos \theta & -\sin \theta & 0 \\ \sin \theta & \cos \theta & 0 \\ 0 & 0 & 1\end{pmatrix}\begin{pmatrix} \cosh \psi & 0 & \sinh \psi \\ 0 & 1 & 0 \\ \sinh \psi & 0 & \cosh\psi\end{pmatrix} \begin{pmatrix} \cos \phi & -\sin \phi & 0 \\ \sin \phi & \cos \phi & 0 \\ 0 & 0 & 1\end{pmatrix}  \]
Now, we request that these diffeomorphisms restrict to the identity map at $r=0$. This means that for all $t \in (-\pi,\pi]$ the following equations must hold:
\[\begin{cases}
\cos \theta \cosh\psi \cos(\phi+t) - \sin \theta \sin(\phi+t)=\left(\ch \psi + \sh \psi\cos(\phi+t) \right)\cos t,\\
\sin\theta \cosh\psi\cos(\phi+t) + \cos\theta \sin(\phi+t)=\left(\ch \psi + \sh \psi\cos(\phi+t) \right)\sin t,
\end{cases}\]
Take $t=-\phi$, multiplying the second equation by $i$ and adding the two equations gives:
\[ e^{i\theta} \cosh \psi = e^\psi e^{-i\phi} \Leftrightarrow \cosh \psi = e^{\psi} e^{-i(\phi+\theta)}. \]
Since the left hand side is positive and real it follows that $\sin(\phi+\theta)=0$ and $\cos(\phi+\theta)\geq 0$, therefore $\phi = - \theta \mod 2\pi$.
It follows also that $\cosh \psi = e^\psi$ which implies $\sinh \psi =0$ and therefore $\psi =0$.
Consequently, $g$ is the identity map.

\begin{exam}
One can remark that Examples~\ref{ExNonRigid} and~\ref{ExRigid} are based on projective conics; the essential difference between them is that in the non-rigid case the boundary is defined by a degenerate conic and, in the rigid case, by a non-degenerate conic. These examples generalise without essential modification to higher dimensions. 

In fact, they also cover the case for the models of the projective compactification of Lorentzian manifolds given by Minkowski spacetime (degenerate case) in which the boundary is $\emph{not}$ boundary rigid. In contrast, the projective compactification of de Sitter spacetime (non-degenerate case) will have this rigidity property. 
\end{exam}

\bigskip

\printbibliography[title=Bibliography]
 \end{document}